\documentclass{amsart}

\usepackage{amsmath,amssymb,amsthm}

\usepackage[nohug,heads=LaTeX]{diagrams}

\newtheorem{prop}{Proposition}[section]
\newtheorem{thm}[prop]{Theorem}
\newtheorem{lem}[prop]{Lemma}
\newtheorem{cor}[prop]{Corollary}

\theoremstyle{definition}

\newtheorem{defn}[prop]{Definition}
\newtheorem{ex}[prop]{Example}

\newtheorem{rem}[prop]{Remark}

\newtheorem*{ack}{Acknowledgements}

\def\co{\colon\thinspace}

\newcommand{\alphac}{\alpha_{\mathrm{c}}}

\newcommand{\C}{\mathbb{C}}
\newcommand{\CP}{\mathbb{C}\mathrm{P}}
\newcommand{\charX}{\mathrm{char}_X}

\newcommand{\rmd}{\mathrm{d}}
\newcommand{\dB}{\rmd_{\mathrm{B}}}

\newcommand{\rme}{\mathrm{e}}

\newcommand{\FF}{\mathcal{F}}

\newcommand{\HB}{H_{\mathrm{B}}}

\newcommand{\rmi}{\mathrm{i}}
\newcommand{\ind}{\mathrm{ind}}

\newcommand{\N}{\mathbb{N}}

\newcommand{\OmegaB}{\Omega_{\mathrm{B}}}
\newcommand{\omegac}{\omega_{\mathrm{c}}}

\newcommand{\tphi}{\tilde{\varphi}}
\newcommand{\tpi}{\tilde{\pi}}

\newcommand{\R}{\mathbb{R}}

\newcommand{\TT}{\mathcal{T}}
\newcommand{\ttheta}{\tilde{\theta}}

\newcommand{\volX}{\mathrm{vol}_X}

\newcommand{\Z}{\mathbb{Z}}

\DeclareMathOperator{\Int}{Int}

\begin{document}

\author{Hansj\"org Geiges}
\address{Mathematisches Institut, Universit\"at zu K\"oln,
Weyertal 86--90, 50931 K\"oln, Germany}
\email{geiges@math.uni-koeln.de}

\thanks{This work is part of a project in the SFB/TRR 191
`Symplectic Structures in Geometry, Algebra and Dynamics',
funded by the DFG}

\title{What does a vector field know
about volume?}

\date{}

\begin{abstract}
This note provides an affirmative answer to a question of Viterbo concerning
the existence of nondiffeomorphic contact forms that share the same Reeb
vector field. Starting from an observation by Croke--Kleiner and Abbondandolo
that such contact forms define the same total volume, we discuss
various related issues for the wider class of geodesible vector fields.
In particular, we define an Euler class of a geodesible vector field
in the associated basic cohomology and give a topological
characterisation of vector fields with vanishing Euler class.
We prove the theorems of Gau{\ss}--Bonnet and Poincar\'e--Hopf
for closed, oriented $2$-dimensional orbifolds
using global surfaces of section and the volume determined
by a geodesible vector field. This volume is computed for Seifert fibred
$3$-manifolds and for some transversely holomorphic flows.
\end{abstract}



\maketitle


\section{Introduction}
This paper is concerned with a question about Reeb flows posed to me
by Claude Viterbo: are there nondiffeomorphic contact forms with the
same Reeb vector field? Viterbo's question was prompted by
Alberto Abbondandolo's discovery of a miraculous identity on differential
forms.

\begin{lem}[Abbondandolo]
Given two differential $1$-forms $\alpha,\beta$ on the same manifold,
the identity
\begin{eqnarray}
\label{eqn:alberto}
\lefteqn{\alpha\wedge(\rmd\alpha)^n-\beta\wedge(\rmd\beta)^n=} \\
 & & (\alpha-\beta)\wedge\sum_{j=0}^n(\rmd\alpha)^j\wedge(\rmd\beta)^{n-j}+
      \rmd\bigl(\alpha\wedge\beta\wedge\sum_{j=1}^{n-1}
      (\rmd\alpha)^j\wedge(\rmd\beta)^{n-1-j}\bigr)\nonumber
\end{eqnarray}
holds for any $n\in\N_0$.
\end{lem}

Identity (\ref{eqn:alberto}), whose verification is straightforward, has the
following striking consequence, which --- as we learned in the
meantime --- has been
observed earlier by Croke and Kleiner~\cite[Lemma~2.1]{crkl94}. They do not
state identity~(\ref{eqn:alberto}), but give a quite similar proof.

\begin{prop}[Croke--Kleiner]
\label{prop:volume}
Let $X$ be a nonsingular vector field on a closed, oriented manifold $M$ of
dimension $2n+1$. Let $\alpha,\beta$ be $1$-forms on $M$ that are
invariant under the flow of $X$ and satisfy
\begin{equation}
\label{eqn:normalise}
\alpha(X)=\beta(X)=1.
\end{equation}
Then
\begin{equation}
\label{eqn:volume}
\int_M\alpha\wedge(\rmd\alpha)^n=\int_M\beta\wedge(\rmd\beta)^n.
\end{equation}
\end{prop}

\begin{proof}
Given (\ref{eqn:normalise}), the invariance condition
$L_X\alpha=L_X\beta=0$ is equivalent to
\begin{equation}
\label{eqn:iXdalpha}
i_X\rmd\alpha=i_X\rmd\beta=0
\end{equation}
by the Cartan formula. Then (\ref{eqn:volume})
is immediate from~(\ref{eqn:alberto}) and Stokes's theorem.
\end{proof}

In particular, this proposition says that any two contact forms
on a closed, oriented manifold that share the same Reeb vector field give rise
to volume forms that integrate to the same total volume. In other
words, this total volume is determined by the Reeb vector field
alone. Abbondandolo has raised the question whether one can compute this 
volume from a given Reeb vector field, not knowing a contact
form it is associated with.

\begin{rem}
Croke and Kleiner used this proposition to conclude that
two compact Riemannian manifolds with $C^1$-conjugate geodesic flows
have the same volume~\cite[Proposition~1.2]{crkl94}. This follows
by considering the canonical contact form on the unit cotangent
bundle, whose Reeb vector field generates the cogeodesic
flow~\cite[Theorem~1.5.2]{geig08}.
\end{rem}

As we shall see, the existence of a $1$-form $\alpha$ as in
Proposition~\ref{prop:volume} is equivalent to the vector field
$X$ being geodesible (Definition~\ref{defn:geodesible},
Proposition~\ref{prop:wadsley}).

\begin{defn}
We write $\volX$ for the real number defined by (\ref{eqn:volume})
and call it the \emph{volume of}~$X$, even though
$\alpha\wedge(\rmd\alpha)^n$ is not, in general, a volume form.
\end{defn}

Much of this paper is a rumination on the consequences and ramifications of
Proposition~\ref{prop:volume}, leading us ultimately towards an
affirmative answer to Viterbo's question (Theorem~\ref{thm:reeb}),
which shows that Proposition~\ref{prop:volume} is indeed a nontrivial
statement, even within the class of Reeb vector fields.
We pay special attention
to the cases where the geodesible vector field
$X$ generates an $S^1$-action, or where
the flow of $X$ admits a global surface of section.
In these cases, one can compute $\volX$
and give it a geometric interpretation.

Along the way, we
introduce the Euler class $e_X$ of a geodesible vector field $X$
in the basic cohomology of the foliation it determines, and we
argue that Proposition~\ref{prop:volume} ought to be interpreted
as a statement in basic cohomology (Proposition~\ref{prop:volume-basic}).
These considerations
will allow us to establish a criterion for the vanishing of $e_X$ in
terms of the existence of a transverse invariant foliation
(Theorem~\ref{thm:vanishing}). Geodesible vector fields $X$ with
$e_X=0$ exist precisely on manifolds that fibre over~$S^1$
(Corollary~\ref{cor:fibration}).

In Section~\ref{section:seifert} we compute $\volX$ for vector fields
that define a Seifert fibration on a $3$-manifold. This computation involves
the use of global surfaces of section. With similar arguments we
prove the theorems of Gau{\ss}--Bonnet and Poincar\'e--Hopf
for closed, oriented $2$-dimensional orbifolds in
Section~\ref{section:gauss}.

For certain geodesible vector fields $X$ whose flow admits
a transverse holomorphic structure, we can relate $\volX$ to
the Bott invariant of that structure. This is the content
of Section~\ref{section:transhol}.

In Section~\ref{section:sos} we derive a formula for $\volX$
when $X$ admits a global surface of section. After presenting the answer to
Viterbo's question in Section~\ref{section:reeb},
we end the paper in Section~\ref{section:orbitequ}
with a brief discussion of orbit equivalent
geodesible vector fields.
\section{Dimension three}
In dimension three, the answer to Viterbo's question is negative.

\begin{prop}
\label{prop:three}
Let $\alpha_0,\alpha_1$ be two contact forms on a closed
$3$-manifold $M$
sharing the same Reeb vector field~$R$. Then $\alpha_0$ and $\alpha_1$
define the same orientation of~$M$. Furthermore, there is an isotopy
$(\psi_t)_{t\in[0,1]}$ of $M$, starting at $\psi_0=\mathrm{id}_M$,
such that $\psi_1^*\alpha_1=\alpha_0$ and $(\psi_t^*)^{-1}\alpha_0$
is a contact form with Reeb vector field $R$ for all $t\in [0,1]$.
\end{prop}

\begin{proof}
The fact that $\alpha_0$ and $\alpha_1$ define the same
orientation of $M$ follows from Proposition~\ref{prop:volume},
since by (\ref{eqn:volume}) the two volume forms $\alpha_i\wedge
\rmd\alpha_i$ must have the same sign.

Set $\alpha_t:=(1-t)\alpha_0+t\alpha_1$, $t\in [0,1]$. Since
$\rmd\alpha_0$ and $\rmd\alpha_1$ restrict to nondegenerate $2$-forms
defining the same orientation
on any tangent $2$-plane field transverse to~$R$, so does $\rmd\alpha_t$.
It follows that $\alpha_t$ is likewise a contact form with Reeb vector
field~$R$. Now apply the Moser trick~\cite[p.~60]{geig08} to the
equation
\begin{equation}
\label{eqn:isotopy3}
\psi_t^*\alpha_t=\alpha_0,
\end{equation}
where we would like the isotopy $(\psi_t)$ to be the flow of a
time-dependent vector field $X_t\in\ker\alpha_t$. Under this
last assumption, by differentiating (\ref{eqn:isotopy3})
we find
\[ \alpha_1-\alpha_0+i_{X_t}\rmd\alpha_t=0,\]
which has a unique solution $X_t\in\ker\alpha_t$.
\end{proof}

Nonetheless, the question how to compute $\mathrm{vol}_R$
for the Reeb vector field $R$ on a closed contact $3$-manifold
$(M,\alpha)$ is extremely interesting.
Cristofaro-Gardiner, Hutchings and Ramos~\cite[Theorem~1.2]{chr15}
have established a deep connection between $\mathrm{vol}_R$
and embedded contact homology (ECH). For a contact $3$-manifold $(M,\alpha)$
with nonzero contact ECH invariant and finite ECH capacities
$c_k(M,\alpha)$, $k\in\N_0$,
the volume of $R$ can be computed as
\[ \mathrm{vol}_R=\lim_{k\rightarrow\infty}\frac{c_k(M,\alpha)^2}{2k}.\]
Through this asymptotic formula,
$\mathrm{vol}_R$ is determined in a subtle way
by the periodic Reeb orbits and their actions.
\section{Geodesible vector fields and taut foliations}
As shown by Wadsley~\cite{wads75}, for a nonsingular vector field
$X$ the existence of a $1$-form $\alpha$ satisfying conditions
(\ref{eqn:normalise}) and (\ref{eqn:iXdalpha}) is equivalent to $X$
being geodesible. Here we briefly recall
the proof of this result, since it is essential to our discussion; see
also~\cite{gluc80,sull78}. Notice that $\volX$ is only defined
for vector fields $X$ on closed manifolds of odd dimension,
but all the considerations about geodesible vector fields in this
and the following two sections make sense,
unless stated otherwise, for manifolds of arbitrary dimension.

\begin{defn}
\label{defn:geodesible}
(a) A nonsingular vector field $X$ on a manifold $M$ is called
\emph{geodesible} if there exists a Riemannian metric on $M$ with
respect to which $X$ has unit length and the flow lines of $X$ are
geodesics.

(b) A $1$-dimensional foliation $\FF$ on a manifold $M$ is
called \emph{taut} if there exists a Riemannian metric on $M$ for
which the leaves of $\FF$ (suitably parametrised) are geodesics.
\end{defn}

\begin{lem}
\label{lem:wadsley}
Let $\bigl(M,\langle\,.\,,\,.\,\rangle\bigr)$ be a Riemannian manifold
with Levi-Civita connection~$\nabla$. Let $X$ be a vector field of unit
length, and set $\alpha=\langle X,\,.\,\rangle$. Then
\[ L_X\alpha=\langle\nabla_X X,\,.\,\rangle.\]
\end{lem}

\begin{proof}
The claimed identity is a pointwise statement.
Locally one can always extend a tangent vector $Y_p\in T_pM$ to 
an $X$-invariant vector field~$Y$, i.e.\ a vector field satisfying
$[X,Y]=0$. Therefore, it suffices to verify the identity
\[ (L_X\alpha)(Y)=\langle\nabla_X X,Y\rangle\]
for such $X$-invariant vector fields~$Y$.
Notice that $\nabla$ being torsion-free
then translates into $\nabla_XY=\nabla_YX$.
Using the fact that the Lie derivative
commutes with contraction, we compute
\begin{eqnarray*}
(L_X\alpha)(Y) & = & L_X(\alpha(Y))-\alpha(L_XY)\;=\;L_X(\alpha(Y))\\
  & = & X\langle X,Y\rangle\\
  & = & \langle\nabla_XX,Y\rangle+\langle X,\nabla_XY\rangle\\
  & = & \langle\nabla_XX,Y\rangle+\langle X,\nabla_YX\rangle\\
  & = & \langle\nabla_XX,Y\rangle+\frac{1}{2}Y\langle X,X\rangle\\
  & = & \langle\nabla_XX,Y\rangle.\qed
\end{eqnarray*}
\renewcommand{\qed}{}
\end{proof}

In the following proposition, the equivalence of (i) with (iv)
is due to Sullivan~\cite{sull78}, who gives an entirely
geometric proof. A more formal proof is given in
\cite[Proposition~6.7]{tond97}; the proof I give is a little more
direct.

\begin{prop}[Wadsley, Sullivan]
\label{prop:wadsley}
Let $X$ be a nonsingular vector field on a manifold~$M$. Then the
following are equivalent:
\begin{itemize}
\item[(i)] $X$ is geodesible;
\item[(ii)] there exists a $1$-form $\alpha$ on $M$ with
$\alpha(X)=1$ and $L_X\alpha=0$;
\item[(iii)] there exists a $1$-form $\alpha$ on $M$ with
$\alpha(X)=1$ and $i_X\rmd\alpha=0$;
\item[(iv)] there is a hyperplane field $\eta$ transverse to $X$ and
invariant under the flow of~$X$.
\end{itemize}
\end{prop}

\begin{proof}
The equivalence of (ii) and (iii) is clear from the Cartan formula.
We first prove the equivalence of (i) and (ii).

Assuming (i), we take $\langle\,.\,,\,.\,\rangle$ to be the metric
for which the flow lines of $X$ are geodesics parametrised by arc length
and set $\alpha=\langle X,\,.\,\rangle$. Then $\nabla_X X=0$,
and (ii) follows from the lemma.

Conversely, given $\alpha$ as in (ii) we choose a metric
$\langle\,.\,,\,.\,\rangle$ on $M$
with $\langle X,X\rangle =1$ and $X\perp\ker\alpha$. Then
$\alpha=\langle X,\,.\,\rangle$, and the vanishing of $L_X\alpha$
implies, by the lemma, that $\nabla_XX=0$.

Next we show the equivalence of (ii) and (iv). Given (ii),
the hyperplane field $\eta:=\ker\alpha$ satisfies (iv). Conversely,
given $\eta$ as in (iv), define a $1$-form $\alpha$ on $M$
by the conditions $\alpha(X)=1$ and $\ker\alpha=\eta$.
Then $i_X\rmd\alpha=L_X\alpha$, and the latter equals $f\alpha$ for
some $f\in C^{\infty}(M)$ by the invariance of~$\eta$. Thus, $i_X\rmd\alpha$
vanishes on~$\eta$. Since $TM=\eta\oplus\langle X\rangle$,
the Lie derivative $L_X\alpha=i_X\rmd\alpha$ vanishes identically.
\end{proof}

\begin{ex}
The Reeb vector field of a contact form or a stable
Hamiltonian structure~\cite{civo15} is geodesible.
\end{ex}

The following characterisation of \emph{oriented} taut $1$-dimensional
foliations, first observed in~\cite{sull78}, is then immediate.
We write $\FF=\langle X\rangle$ with any nonsingular vector field
$X$ whose flow lines are the leaves of~$\FF$.

\begin{prop}
\label{prop:sullivan}
The oriented $1$-dimensional foliation $\FF=\langle X\rangle$ is taut
if and only if there is a $1$-form $\alpha$ on $M$ with
$\alpha(X)>0$ and $i_X\rmd\alpha=0$.
\end{prop}

\begin{proof}
If $\FF=\langle X\rangle$ is taut, rescale $X$ to a vector field
of length $1$ with respect to the metric that makes the leaves of $\FF$
geodesics. Then the existence of the desired $1$-form $\alpha$
follows from the equivalence of (i) and (iii) in
Proposition~\ref{prop:wadsley}.

Conversely, given $\alpha$, the rescaled vector field
$X/\alpha(X)$, which likewise spans~$\FF$,
satisfies (iii) in Proposition~\ref{prop:wadsley}.
\end{proof}

\begin{rem}
(1) Alternatively, one can derive the equivalence of (i) and (iii)
in Proposition~\ref{prop:wadsley} from the identity
\[ i_X\rmd\alpha=
\langle\nabla_XX,\,.\,\rangle-\rmd\bigl(\langle X,X\rangle/2\bigr),\]
where again $\alpha=\langle X,\,.\,\rangle$; this identity holds for
any vector field~$X$, see~\cite[Section~2.3]{civo15}.

(2) The main point of Sullivan's article \cite{sull78} is a characterisation
of taut foliations in terms of the absence of ``tangent homologies''.
I refer to \cite{gluc80} for a beautiful discussion of Sullivan's
theorem; there one can find examples of $1$-dimensional
oriented foliations that are not taut.
\end{rem}
\section{Basic cohomology}
Here are the elementary notions of basic differential forms
and basic cohomology associated with a foliation. I restrict
attention to oriented $1$-dimensional foliations $\FF=\langle X\rangle$;
for a more comprehensive treatment see~\cite[Chapter~4]{tond97}.

\begin{defn}
A differential form $\omega$ on $(M,\FF)$ is called \emph{basic} if
\[ i_X\omega=0\;\;\;\text{and}\;\;\; i_X\rmd\omega =0.\]
\end{defn}

Notice that this definition does not depend on the choice of
vector field $X$ spanning~$\FF$.
We write $\OmegaB^k(\FF)$ for the vector space of basic $k$-forms
on $(M,\FF)$. The usual exterior differential $\rmd$ restricts to
\[ \dB\co\OmegaB^k\longrightarrow\OmegaB^{k+1},\]
and the \emph{basic cohomology groups} $\HB^k(\FF)$ are defined as
the cohomology groups of the complex $\bigl(\OmegaB^{\bullet}(\FF),\dB\bigr)$.
The cohomology class of a $k$-form $\omega\in\ker\dB$ is written as
$[\omega]_{\mathrm{B}}\in\HB^k(\FF)$.

The following definitions are motivated by Propositions \ref{prop:wadsley}
and~\ref{prop:sullivan}. The notation $C_X,C_{\FF}$ is chosen
because the $1$-form $\alpha=\langle X,\,.\,\rangle$ (with $X$ of unit length)
is the \emph{characteristic form} of~$\FF$~\cite[p.~69]{tond97}
with respect to the metric $\langle\,.\,,\,.\,\rangle$. We adapt this
definition to the case of geodesible vector fields, where it is reasonable
to consider only those metrics for which the flow lines of $X$ are geodesics.

\begin{defn}
Let $X$ be a geodesible vector field. Any $1$-form $\alpha$ with
$\alpha(X)=1$ and $i_X\rmd\alpha=0$ is called a \emph{characteristic $1$-form}
of~$X$. We write $\charX$ for the space of these characteristic forms.
\end{defn}

\begin{defn}
(a) Let $X$ be a geodesible vector field on a manifold~$M$. Set
\[ \Omega_X^1:=\bigl\{\alpha\in\Omega^1(M)\co
\text{$\alpha(X)=c$ for some $c\in\R^+$},\; i_X\rmd\alpha=0\bigr\}\]
and
\[ C_X:=\Omega_X^1/\!\sim,\]
where
\[ \alpha\sim\beta:\Longleftrightarrow\alpha(X)=\beta(X).\]
The equivalence class of $\alpha\in\Omega_X^1$ is written as
$[\alpha]_X\in C_X$. Obviously there is a canonical
identification of $C_X$ with~$\R^+$.

(b) Let $\FF=\langle X\rangle$ be an oriented taut $1$-dimensional
foliation on~$M$. Set
\[ \Omega_{\FF}^1:=\bigl\{\alpha\in\Omega^1(M)\co
\alpha(X)>0,\; i_X\rmd\alpha=0\bigr\}\]
and
\[ C_{\FF}:=\Omega_{\FF}^1/\!\sim,\]
where the equivalence relation $\sim$ is defined as in~(a). The
equivalence class of $\alpha\in\Omega_{\FF}^1$ is written as
$[\alpha]_{\FF}$. Notice that these definitions do not depend
on the choice of~$X$.
\end{defn}

The assumptions on geodesibility and tautness, respectively, guarantee
that we are not talking about empty sets.

The spaces $\Omega_X^1$, $\charX$ and $\Omega_{\FF}^1$ are obviously convex.
The proof of Proposition~\ref{prop:wadsley} shows that, for a geodesible
vector field~$X$, the map
\[ \begin{array}{ccc}
\mathrm{Met}_X             & \longrightarrow & \charX\\
\langle\,.\,,\,.\,\rangle & \longmapsto     & \alpha=\langle X,\,.\,\rangle
\end{array} \]
from the space $\mathrm{Met}_X$ of metrics for which $X$ has unit length
and geodesic flow lines is a Serre fibration with fibre the space of metrics
on a hyperplane field transverse to~$X$, which can be seen
as follows. Given a family $\alpha_{q,t}\in \charX$,
where $t\in [0,1]$ and $q$ varies in some parameter space,
and a family of metrics $\langle\,.\,,\,.\,\rangle_{q,0}$
with $\langle X,\,.\,\rangle_{q,0}=\alpha_{q,0}$, one simply defines
$\langle\,.\,,\,.\,\rangle_{q,t}$ by the following requirements:
\begin{itemize}
\item[(i)] $\langle X,X\rangle_{q,t}=1$;
\item[(ii)] $\ker\alpha_{q,t}\perp X$ with respect to
$\langle\,.\,,\,.\,\rangle_{q,t}$;
\item[(iii)] $\langle\,.\,,\,.\,\rangle_{q,t}|_{\ker\alpha_{q,t}}=
\langle\,.\,,\,.\,\rangle_{q,0}|_{\ker\alpha_{q,0}}$ under the identification
of $\ker\alpha_{q,t}$ with $\ker\alpha_{q,0}$ given by projection along~$X$.
\end{itemize}
Of course, this Serre fibration
property is not terribly useful, since all spaces in question
are contractible.

\begin{prop}
\label{prop:pairing}
Let $M$ be a closed, oriented manifold of dimension~$m$.

\begin{itemize}
\item[(a)] Let $X$ be a geodesible vector field on~$M$.
Set $\FF=\langle X\rangle$. Then the map
\[ \begin{array}{ccccl}
C_X              & \times & \HB^{m-1}(\FF)             
     & \longrightarrow & \R\\
\bigl([\alpha]_X & ,      & [\sigma]_{\mathrm{B}}\bigr)
     & \longmapsto     & [\alpha]_X\bullet[\sigma]_{\mathrm{B}}:=
                         \int_M\alpha\wedge\sigma
\end{array} \]
is well defined.
\item[(b)] Let $\FF$ be an oriented taut
$1$-dimensional foliation on~$M$. Then the map
\[ \begin{array}{ccccl}
C_{\FF}              & \times & \HB^{m-1}(\FF)
     & \longrightarrow & \R\\
\bigl([\alpha]_{\FF} & ,      & [\sigma]_{\mathrm{B}}\bigr)
     & \longmapsto     & [\alpha]_{\FF}\bullet[\sigma]_{\mathrm{B}}:=
                         \int_M\alpha\wedge\sigma
\end{array} \]
is well defined.
\end{itemize}
\end{prop}

\begin{proof}
We prove (b); the proof of (a) is completely analogous.
Write $\FF=\langle X\rangle$ with some nonsingular vector field $X$
spanning~$\FF$.

(i) We have $i_X\sigma=0$, since $\sigma\in\OmegaB^{m-1}(\FF)$.
Suppose $[\alpha]_{\FF}=[\alpha']_{\FF}$, which means that the function
$\alpha(X)-\alpha'(X)$ is identically zero. It follows that
the $m$-form $(\alpha-\alpha')\wedge\sigma$ vanishes identically,
since its interior product with the nonsingular vector field $X$ vanishes.

(ii) Suppose $[\sigma]_{\mathrm{B}}=[\sigma']_{\mathrm{B}}\in
\HB^{m-1}(\FF)$, that is, $\sigma-\sigma'=\rmd\tau$ for some
$\tau\in\OmegaB^{m-2}(\FF)$. Then
\begin{eqnarray*}
\int_M\alpha\wedge(\sigma-\sigma')
   & = & \int_M\alpha\wedge\rmd\tau\\
   & = & -\int_M\rmd(\alpha\wedge\tau)+\int_M\rmd\alpha\wedge\tau.
\end{eqnarray*}
The first summand vanishes by Stokes's theorem; the integrand
of the second summand vanishes identically, since
$i_X(\rmd\alpha\wedge\tau)=0$.
\end{proof}

Observe that the maps defined in this proposition are positively
homogeneous of degree $1$ on the first factor, and linear in the
second factor.
\section{The Euler class of a geodesible vector field}
Let $X$ be a geodesible vector field on a manifold $M$ and set
$\FF=\langle X\rangle$. Choose a characteristic $1$-form $\alpha$ for~$X$.

\begin{lem}
The basic cohomology class $e_X:=-[\rmd\alpha]_{\mathrm{B}}\in\HB^2(\FF)$
is determined by~$X$.
\end{lem}

\begin{proof}
Let $\beta$ be a further characteristic $1$-form.
Then $\gamma:=\alpha-\beta\in\OmegaB^1(\FF)$, and
$\rmd\alpha-\rmd\beta=\rmd\gamma=\dB\gamma$, hence
$[\rmd\alpha]_{\mathrm{B}}=[\rmd\beta]_{\mathrm{B}}$.
\end{proof}

I do not know whether the following definition has been made before,
but it is certainly a very natural one.

\begin{defn}
The class $e_X\in\HB^2(\FF)$ is called the
\emph{Euler class} of the geodesible vector field~$X$.
\end{defn}

\begin{ex}
(1) If the flow of $X$ generates a principal $S^1$-action,
where we think of $S^1$ as $\R/\Z$, then $e_X$ can be naturally
identified with the real Euler class $e\otimes\R\in H^2(M/S^1;\R)$
of the $S^1$-bundle $M\rightarrow M/S^1$. Our definition accords
with the usual sign convention, cf.\
\cite[Section~6.2]{mori01}, \cite[Section~7.2]{geig08}.

(2) If the flow of $X$ generates a locally free $S^1$-action, then
$\HB^{\bullet}(\FF)$ may be thought of as the orbifold cohomology
of the orbifold $M/S^1$, and $e_X$ as the real Euler class of the
$S^1$-orbibundle $M\rightarrow M/S^1$. We discuss examples of this
kind in detail in Sections \ref{section:seifert}
and~\ref{section:gauss}. For more information on $S^1$-orbibundles
in the general sense see~\cite{kela20}.
\end{ex}

We shall meet further examples in Section~\ref{section:sos}, where we discuss
surfaces of section for the flow of~$X$.

The next lemma is the generalisation of a result for
connection $1$-forms of principal $S^1$-bundles.

\begin{lem}
Let $X$ be a geodesible vector field and $\omega\in\OmegaB^2(\FF)$
a basic $2$-form with $-[\omega]_{\mathrm{B}}=e_X$. Then there
is a characteristic $1$-form $\beta$ with $\rmd\beta=\omega$.
\end{lem}

\begin{proof}
Since $[\omega]_{\mathrm{B}}=[\rmd\alpha]_{\mathrm{B}}$, we find
a basic $1$-form $\gamma\in\OmegaB^1(\FF)$ with $\omega=\rmd\alpha+
\rmd\gamma$. Then $\beta:=\alpha+\gamma$ is the desired characteristic form.
\end{proof}

This lemma implies the following proposition.

\begin{prop}
\label{prop:symplectic}
A geodesible vector field $X$ on a manifold $M$ of dimension $2n+1$
is the Reeb vector field of a contact form
if and only if the Euler class $e_X$ has an odd-symplectic representative,
i.e.\ if there is a closed basic $2$-form $\omega\in\Omega^2_{\mathrm{B}}(M)$
with $-[\omega]_{\mathrm{B}}= e_X$ and $\omega^n\neq 0$.
\qed
\end{prop}

The following expression of the volume $\volX$ in terms of the
Euler class is immediate from the definitions.
This is the promised cohomological interpretation and generalisation
of Proposition~\ref{prop:volume}.

\begin{prop}
\label{prop:volume-basic}
Let $X$ be a geodesible vector field on a closed, oriented manifold $M$ of
dimension $2n+1$, and $\alpha$ a characteristic form for~$X$. Then
\[ \volX=(-1)^n[\alpha]_X\bullet e_X^n.\]
If $X$ generates a free $S^1$-action, we have ---
with $e\in H^2(B;\Z)$ denoting the Euler class of the fibration
$M\rightarrow M/S^1=:B$ ---
\[ \volX=(-1)^n\langle e^n,[B]\rangle,\]
where $[B]$ denotes the fundamental class of $B$ and
$\langle\,.\,,\,.\,\rangle$ the Kronecker pairing.
\qed
\end{prop}

Here is a useful vanishing criterion for the Euler class.
For the flow of $X$ to be globally defined,
we assume $M$ to be closed.

\begin{thm}
\label{thm:vanishing}
The Euler class $e_X\in\HB^2(\FF)$
of a geodesible vector field $X$ on a closed manifold $M$
vanishes if and only if $X$ admits a transverse foliation $\TT$
invariant under the flow of~$X$.
\end{thm}

\begin{proof}
Suppose that $e_X=0$. As before we
write $\FF=\langle X\rangle$.
Choose a $1$-form $\alpha$ with $\alpha(X)=1$ and $i_X\rmd\alpha=0$.
Then $[\rmd\alpha]_{\mathrm{B}}=-e_X=0$, so there is a basic
$1$-form $\gamma\in\OmegaB^1(\FF)$ with $\rmd\gamma=\rmd\alpha$.
Then $\beta:=\alpha-\gamma$ is a closed $1$-form with $\beta(X)=1$.
In particular, $\ker\beta$ defines a foliation $\TT$ transverse to~$X$, and
$\TT$ is invariant under the flow of $X$ since
$L_X\beta=\rmd(\beta(X))+i_X\rmd\beta=0$.

Conversely, let $\TT$ be a transverse invariant foliation. Define
a $1$-form $\alpha$ by $\alpha(X)=1$ and $\ker\alpha=T\TT$, where
$T\TT$ denotes the distribution of tangent spaces to~$\TT$. 
Then $i_X\rmd\alpha=L_X\alpha$, and the latter equals $f\alpha$
for some $f\in C^{\infty}(M)$ by the invariance of~$\TT$.
This implies $\rmd\alpha(X,Y)=0$ for $Y\in\Gamma(T\TT)$.

Given two (local) vector fields $Y_1,Y_2\in\Gamma(T\TT)$, we compute
\[ \rmd\alpha(Y_1,Y_2)=Y_1\alpha(Y_2)-Y_2\alpha(Y_1)-
\alpha([Y_1,Y_2])=0.\]
Thus, we conclude that $\rmd\alpha=0$, and hence $e_X=0$.
\end{proof}

\begin{cor}
\label{cor:fibration}
A closed manifold $M$ admits a geodesible vector field $X$
with $e_X=0$ if and only if $M$ fibres over~$S^1$.
\end{cor}

\begin{proof}
If $M$ admits a geodesible vector field with $e_X=0$, the fact that
$M$ fibres over $S^1$
follows from the existence of a closed, nonsingular
$1$-form on $M$, established in the foregoing proof,
and a result of Tischler~\cite{tisc70},
cf.~\cite[Section~9.3]{conl01}.

Conversely, a manifold $M$ that fibres over $S^1$ always admits a geodesible
vector field~\cite{gluc80}. Such a manifold $M$ can be written
as $[0,1]\times F/(1,x)\sim(0,\psi(x))$, where $F$ denotes
the fibre and $\psi$ the monodromy of the bundle. Let $g_{\theta}$,
$\theta\in[0,1]$ be any smooth family of metrics on $F$ with $\psi^*g_0=g_1$.
Then $d\theta^2+g_{\theta}$ defines a metric on $[0,1]\times F$ for which the
segments $[0,1]\times\{x\}$ are geodesics, and this metric descends to~$M$.

Alternatively, let $\alpha$ be the pull-back of the $1$-form $\rmd\theta$
under the bundle projection $M\rightarrow S^1$. Then $\rmd\alpha=0$,
so any vector field $X$ on $M$ with $\alpha(X)=1$, i.e.\
any lift of~$\partial_{\theta}$, is geodesible, and clearly $e_X=0$.
\end{proof}

\begin{ex}
For Seifert fibred $3$-manifolds (see the next section), the statement of
Corollary~\ref{cor:fibration} can be found in
\cite[Theorem~5.4]{scot83}.
\end{ex}
\section{Seifert fibred $3$-manifolds}
\label{section:seifert}
In this section we take $M\rightarrow B$ to be a Seifert fibration
of a closed, oriented $3$-manifold $M$ over a closed, oriented
$2$-dimensional orbifold~$B$.
Let $X$ be the vector field whose flow defines an $S^1$-action
on $M$ with orbits equal to the Seifert fibres, where the minimal
period of the regular fibres is assumed to be equal to~$1$.
I refer to~\cite{gela18} and \cite{jane83}
for the basic terminology of Seifert fibrations.

Suppose the Seifert invariants of $M\rightarrow B$ are
\[ \bigl( g;(\alpha_1,\beta_1),\ldots,(\alpha_n,\beta_n)\bigr),\]
where $g\in\N_0$ is the genus of $B$, and the
$(\alpha_i,\beta_i)$, $i=1,\ldots, n$, are pairs of coprime
integers with $\alpha_i\neq 0$. Here the $\alpha_i$ give
the multiplicities of the singular fibres; the pairs with
$\alpha_i=1$ do not correspond to singular fibres, but contribute to
the Euler class of the fibration.

Concretely, $M$ is recovered from these Seifert invariants as
follows. Let $B$ be the closed, oriented surface of genus~$g$,
and remove $n$ disjoint discs to obtain
\[ B_0=B\setminus\Int\bigl(D^2_1\sqcup\ldots\sqcup D^2_n\bigr).\]
Over this surface with boundary, we take the
trivial $S^1$-bundle $M_0=B_0\times S^1\rightarrow B_0$.
Write the boundary $\partial B_0$ with the opposite of its natural
orientation as
\[ -\partial B_0=S^1_1\sqcup\ldots\sqcup S^1_n.\]
We write the fibre class of this trivial fibration as $h=\{*\}\times S^1$,
and on $\partial M_0$ we consider the curves
\[ q_i=S^1_i\times\{0\},\; i=1,\ldots,n;\]
recall that we think of the fibre $S^1$ as $\R/\Z$.
The labels $h,q_1,\ldots,q_n$ should be read as isotopy classes of curves
on $\partial M_0$.

Choose integers $\alpha_i',\beta_i'$, $i=1,\ldots,n$, such that
\[ \begin{vmatrix}
\alpha_i & \alpha_i'\\
\beta_i  & \beta_i'
\end{vmatrix}=1.\]
Further, take $n$ copies $V_i=D^2\times S^1$ of a solid torus,
where $D^2$ is the unit disc in~$\R^2$, with
respective meridian and longitude
\[ \mu_i=\partial D^2\times\{0\},\;\; \lambda_i=\{1\}\times S^1\subset
\partial V_i.\]
Then glue the $V_i$ to $M_0$ along the boundary, where $\partial V_i$
is identified with the component $S_i^1\times S^1$ of $\partial M_0$ via
\begin{equation}
\label{eqn:gluing}
h=-\alpha_i'\mu_i+\alpha_i\lambda_i,\;\;\;
q_i=\beta_i'\mu_i-\beta_i\lambda_i.
\end{equation}
Notice that the fibration of $M_0$ given by the fibre class
$h$ extends to a fibration of $V_i$ with the central fibre of
multiplicity~$\alpha_i$. This is the description of $M\rightarrow B$
with the given Seifert invariants.

The Euler number $e$ of the Seifert fibration with the given Seifert
invariants, defined as the obstruction to the existence of a
section (in the Seifert sense)~\cite[Section~3]{jane83}, is
\[ e=-\sum_{i=1}^n\frac{\beta_i}{\alpha_i}.\]
We now want to use a global surface of section (in a slightly
generalised sense) to derive this formula.

\begin{prop}
\label{prop:seifert-euler}
Let $X$ be a vector field on a closed, oriented $3$-manifold
$M$ defining a Seifert fibration of regular period~$1$ with invariants
\[ \bigl( g;(\alpha_1,\beta_1),\ldots,(\alpha_n,\beta_n)\bigr).\]
Then
\[ \langle e_X,[B]\rangle=-\sum_{i=1}^n\frac{\beta_i}{\alpha_i},\]
where $\langle\,.\,,\,.\,\rangle$ denotes the Kronecker pairing
between $\HB^2(\FF)$ and $H_2(B)$.
\end{prop}

Recall that a \emph{global surface of section} (s.o.s.)\ for the
flow of $X$ is an embedded compact surface $\Sigma\subset M$
whose boundary consists of orbits of~$X$, whose interior $\Int(\Sigma)$
is transverse to $X$, and such that
the flow line of $X$ through any point not on $\partial\Sigma$
hits $\Int(\Sigma)$ in forward and backward time.
We now describe such an s.o.s.\ for the situation at hand.

\begin{proof}[Proof of Proposition~\ref{prop:seifert-euler}]
In $M_0$ we can take $B_0\times\{0\}$ as section. The boundary
of this section consists of the curves $-q_i$, $i=1,\ldots,n$,
which are identified with $-\beta_i'\mu_i+\beta_i\lambda_i$
on $\partial V_i$. Thus, by isotoping these respective curves
radially towards the spine $\sigma_i=\{0\}\times S^1$ of~$V_i$,
we sweep out a surface $\Sigma$ that is not quite an s.o.s.\
in the sense of the definition above, but which has the following
properties:
\begin{itemize}
\item[-] the inclusion $\Sigma\subset M$ is an embedding on 
$\Int(\Sigma)$;
\item[-] the boundary of $\Sigma$ is made up of the curves $\sigma_i$,
each covered $\beta_i$ times;
\item[-] the interior $\Int(\Sigma)$ is intersected positively in
a single point by each $X$-orbit different from the~$\sigma_i$.
\end{itemize}
We now choose a specific connection $1$-form $\alpha$ on $M\rightarrow B$,
i.e.\ a characteristic $1$-form for~$X$. On $V_i=D^2\times S^1$
we write $(r,2\pi\phi)$ for polar coordinates on the $D^2$-factor,
and $\theta\in S^1=\R/\Z$. It follows from (\ref{eqn:gluing})
that on $V_i$ we may assume $X$ to be given by
$-\alpha_i'\partial_{\phi}+\alpha_i\partial_{\theta}$.
So we choose the $1$-form $\alpha$
equal to $\alpha=\rmd\theta/\alpha_i$ near the spine of $V_i$, and
then extend arbitrarily as a connection form over~$M$ (using a partition
of unity).

We then compute
\[ \langle e_X,[B]\rangle  =  -\int_B\rmd\alpha
= -\int_{\Sigma}\rmd\alpha
= -\int_{\partial\Sigma}\alpha
= -\sum_{i=1}^n\frac{\beta_i}{\alpha_i}. \]
Notice that the integral $\int_B\rmd\alpha$ is well defined,
since $\rmd\alpha$ is a basic form.
\end{proof}

\begin{rem}
As in Proposition~\ref{prop:symplectic} one argues that if the
Euler number $e$ of the Seifert fibration is nonzero, one can choose
$\alpha$ as a contact form (defining the correct orientation of~$M$
if $e<0$, the opposite one if $e>0$).
\end{rem}

\begin{cor}
The volume $\volX$ of a vector field $X$ defining a Seifert fibration
on a closed, orientable $3$-manifold, with the regular
fibres having minimal period~$1$,
equals minus the Euler number of that Seifert fibration.
In particular, with $m$ denoting the least common multiple
of the multiplicities $\alpha_1,\ldots,\alpha_n$, we have that
$m\cdot\volX$ is an integer.
\end{cor}

\begin{proof}
The value of the integral of $\alpha\wedge\rmd\alpha$ over $M$ does not
change when we remove the singular fibres of the Seifert fibration. But then
the integral equals
\[ \int_{\Int(\Sigma)\times S^1}\alpha\wedge\rmd\alpha=\int_{\Sigma}\rmd\alpha
=-e.\qed\]
\renewcommand{\qed}{}
\end{proof}

\begin{rem}
The integrality statement has been observed in greater
generality by Weinstein~\cite{wein77}.
\end{rem}

\begin{ex}
The positive Hopf fibration
\[ \begin{array}{rcccl}
\C^2\;\supset & S^3       & \longrightarrow & \CP^1 & =\; S^2\\
            & (z_1,z_2) & \longmapsto     & [z_1:z_2]
\end{array}\]
is given by the vector field
$X=2\pi(\partial_{\varphi_1}+\partial_{\varphi_2})$ of period~$1$,
where $\varphi_1,\varphi_2\in\R/2\pi\Z$.
The corresponding connection $1$-form is
\[ \alpha=\frac{1}{2\pi}\,\bigl(r_1^2\,\rmd\varphi_1
+r_2^2\,\rmd\varphi_2\bigr).\]
With $r^2=r_1^2+r_2^2$ one computes
\[ r\,\rmd r\wedge\alpha\wedge\rmd\alpha=\frac{1}{2\pi^2}
(r_1^2+r_2^2)\cdot (r_1\,\rmd r_1\wedge\rmd\varphi_1\wedge r_2\,\rmd r_2
\wedge\rmd\varphi_2).\]
So along the unit sphere $S^3=\{r=1\}$, the $3$-form $\alpha\wedge\rmd\alpha$
restricts to the standard volume form up to a factor $1/2\pi^2$, hence
\[ \int_{S^3}\alpha\wedge\rmd\alpha=\frac{1}{2\pi^2}\mathrm{Vol}(S^3)=1.\]

A section of the Hopf fibration over $\C\cong\CP^1\setminus\{[0:1]\}$
is defined by
\[ r\rme^{\rmi\varphi}\longmapsto \bigl[1:\rme^{\rmi\varphi}\bigr]
\longmapsto \Bigl(\frac{1}{\sqrt{1+r^2}},
\frac{r\rme^{\rmi\varphi}}{\sqrt{1+r^2}}\Bigr).\]
Under this map, $\rmd\alpha$ pulls back to
\[ \frac{1}{\pi}\cdot\frac{r}{(1+r^2)^2}\,\rmd r\wedge\rmd\varphi.\]
This yields
\[ \int_{S^2}\rmd\alpha=\frac{1}{\pi}\,\int_0^{2\pi}\int_0^{\infty}
\frac{r}{(1+r^2)^2}\,\rmd r\,\rmd\varphi=1.\]

Thus, both computations confirm that the positive Hopf fibration has Euler
number $e=-1$, see also \cite[Lemma~2.2]{agz19}.
\end{ex}
\section{The theorems of Gau{\ss}--Bonnet and Poincar\'e--Hopf}
\label{section:gauss}
In this section we formulate and prove the theorems of Gau{\ss}--Bonnet
and Poin\-car\'e--Hopf for oriented $2$-dimensional orbifolds, using
an s.o.s.\ argument as in the preceding section. Versions of these
theorems for higher-dimen\-sio\-nal orbifolds (including those with boundary)
can be found in \cite{sata57} and~\cite{seat08}. In order to avoid confusion
with formulas found elsewhere, in this section we follow the
usual convention that the regular fibres in the unit tangent bundle of
a $2$-dimensional orbifold have length~$2\pi$.

Thus, let $B$ be a closed, oriented $2$-dimensional Riemannian orbifold
with underlying surface of genus $g$ and $n$ cone points of multiplicities
$\alpha_1,\ldots,\alpha_n$. There is
a well-defined unit tangent bundle $STB$, cf.~\cite{gego12}, which is
a Seifert manifold with invariants
\[ \bigl(g,(1,2g-2),(\alpha_1,\alpha_1-1),\ldots,(\alpha_n,\alpha_n-1\bigr).\]
The orbifold Euler characteristic $\chi_{\mathrm{orb}}(B)$
is the Euler number of the Seifert fibration
$\pi\co STB\rightarrow B$, so by Proposition~\ref{prop:seifert-euler} we
have
\[ \chi_{\mathrm{orb}}(B)=2-2g-n+\sum_{i=1}^n\frac{1}{\alpha_i}.\]
This formula can also be derived combinatorially, using the
Riemann--Hurwitz formula for coverings, see~\cite[p.~427]{scot83}.

Just like the unit tangent bundle of a smooth
surface, the unit tangent bundle $STB$ of an orbifold
admits a pair of Liouville--Cartan forms $\lambda_1,\lambda_2$ and
a connection $1$-form $\tilde{\alpha}$ satisfying the structure equations
\begin{eqnarray*}
\rmd\lambda_1      & = & -\lambda_2\wedge\tilde{\alpha},\\
\rmd\lambda_2      & = & -\tilde{\alpha}\wedge\lambda_1,\\
\rmd\tilde{\alpha} & = & -(\pi^*K)\lambda_1\wedge\lambda_2,
\end{eqnarray*}
where $K$ is the Gau{\ss} curvature of the Riemannian
metric on~$B$. See \cite[Section~2.1]{agz18}
for the surface case, and \cite[Section~7]{gego95} for a discussion
of Liouville--Cartan forms for orbifolds.

\begin{thm}[Gau{\ss}--Bonnet]
The total curvature of a closed, oriented
$2$-dimen\-sio\-nal Riemannian orbifold $B$ equals
\[ \int_B K\,\rmd A=2\pi\chi_{\mathrm{orb}}(B).\]
\end{thm}

\begin{proof}
The characteristic $1$-form $\alpha$ for the vector field $X$ that makes the
regular fibres of $STB$ of length $1$ is $\alpha=\tilde{\alpha}/2\pi$.
Therefore, with $e=\chi_{\mathrm{orb}}(B)$ we obtain
\[ \int_B K\,\rmd A=-\int_{\Sigma}\rmd\tilde{\alpha}=
-2\pi\int_{\Sigma}\rmd\alpha=2\pi\chi_{\mathrm{orb}}(B),\]
where $\Sigma$ is as in the proof of Proposition~\ref{prop:seifert-euler}.
\end{proof}

Now let $Y$ be a vector field with isolated zeros
on the orbifold~$B$. In order to
formulate the Poincar\'e--Hopf theorem we need to give a definition
of the index $\ind_{p}Y$ in an orbifold singularity $p\in B$,
cf.~\cite[Section~3.2]{sata57}.
First, let $p\in B$ be a smooth point where $Y$ has a zero. Choose
a small disc $D_{\varepsilon}(p)\subset B$ not containing
other zeros of~$Y$ (and hence, as we shall see,
in particular no orbifold points of~$B$). Choose
a trivialisation $TD_{\varepsilon}(p)\cong D_{\varepsilon}(p)\times\R^2$.
Then $\ind_pY$ is the degree of the map $\partial D_{\varepsilon}(p)
\rightarrow S^1$, $x\mapsto Y(x)/|Y(x)|$.

When the zero of $Y$ happens to be an orbifold singularity $p_i\in B$
of order~$\alpha_i$, we consider a local description $\pi_{\alpha_i}\co
D^2\rightarrow D^2/\Z_{\alpha_i}\cong D^2$ of the singularity, where the
cyclic group $\Z_{\alpha_i}$ is generated by the rotation about $0\in\R^2$
through an angle $2\pi/\alpha_i$.

We drop the index $i$ for the time being; there should
be little grounds for confusing $\alpha$ in the following discussion
with the connection $1$-form.

The fibre of $STB$ over the singular point $p\in B$ has Seifert
invariants $(\alpha,\beta=\alpha-1)$, so we may take
$\alpha'=\beta'=1$. Then, cf.~\cite{gela18},
the local description of the fibration $STB\rightarrow B$ near
the orbifold point $p$ is given by
\[ \begin{array}{rccc}
\pi\co & D^2\times S^1                                     & \longrightarrow
  & D^2\\
       & \bigl(r\rme^{\rmi\varphi},\rme^{\rmi\theta}\bigr) & \longmapsto
  & r\rme^{\rmi(\alpha\varphi+\theta)},
\end{array} \]
where we identify $p$ with $0\in D^2$.
Notice that the fibres of $\pi$ are described by $\alpha\varphi+\theta
=\mathrm{const.}$, or in parametric form as
\[ t\longmapsto \bigl(\varphi(t),\theta(t)\bigr)=
(\varphi_0-t,\theta_0+\alpha t).\]
This accords with (\ref{eqn:gluing}).

Now consider the following commutative diagram,
\begin{diagram}
D^2\times S^1 & \ni & \bigl(r\rme^{\rmi\tphi},\rme^{\rmi\ttheta}\bigr)
&& \rMapsto^{\tpi_{\alpha}} 
&& \bigl(r\rme^{\rmi\varphi},\rme^{\rmi\theta}\bigr) & \in &  D^2\times S^1\\
\dTo^{\tpi} && \dMapsto && && \dMapsto && \dTo_{\pi}\\
D^2 & \ni           & r\rme^{\rmi(\tphi+\ttheta)}
&& \rMapsto^{\pi_{\alpha}}
&& r\rme^{\rmi(\alpha\varphi+\theta)} & \in & D^2,
\end{diagram}
where the quotient map
$\pi_{\alpha}$ under the $\Z_{\alpha}$-action is
given by
\[ \pi_{\alpha}(r\rme^{\rmi t})=r\rme^{\rmi\alpha t},\]
its lift $\tpi_{\alpha}$ to the unit tangent bundle by
\[ \tpi_{\alpha}\co (\tphi,\ttheta)\longmapsto (\varphi,\theta)=(\tphi,
\alpha\ttheta).\]

Up to homotopy, the section $Y/|Y|$ of $\pi$ over $\partial D^2$
is of the form
\begin{equation}
\label{eqn:curve}
\bigl(\varphi(t),\theta(t)\bigr)=\bigl(2\pi kt,2\pi(1-k\alpha)t\bigr),
\;\;\; t\in [0,1],
\end{equation}
for some $k\in\Z$; notice that $\alpha\varphi(t)+\theta(t)$ goes
from $0$ to~$2\pi$ as $t$ goes from $0$ to~$1$.
The lift of the $\alpha$-fold traversal of this curve under the map
$\tpi_{\alpha}$ is described by
\begin{equation}
\label{eqn:liftedcurve}
\bigl(\tphi(t),\ttheta(t)\bigr)=\bigl(2\pi kt,
2\pi\frac{1-k\alpha}{\alpha}\,t\bigr),
\;\;\; t\in [0,\alpha];
\end{equation}
here $\tphi(t)+\ttheta(t)=\frac{2\pi}{\alpha}\,t$ goes from $0$ to $2\pi$
as $t$ goes from $0$ to~$\alpha$.

The fibres of $\tpi$ are described by $\tphi+\ttheta=\mathrm{const.}$,
and a single right-handed Dehn twist along a meridional disc
of $D^2\times S^1$,
\[ (\tphi',\ttheta'):=(\tphi+\ttheta,\ttheta),\]
will bring these fibres into the form $\tphi'=\tphi_0'$, and
the curve (\ref{eqn:liftedcurve}) becomes
\[ \bigl(\tphi'(t),\ttheta'(t)\bigr)=\Bigl(\frac{2\pi}{\alpha}\, t,
2\pi\frac{1-k\alpha}{\alpha}\,t\Bigr),
\;\;\; t\in [0,\alpha].\]

The index $\ind_p Y$ at an orbifold point of multiplicity $\alpha$
is defined as $\ind_{\tilde{p}}\widetilde{Y}/\alpha$, where $\tilde{p}=
\pi_{\alpha}^{-1}(p)$ and $\widetilde{Y}$ is the lifted vector field.
Our considerations show that, in dependence on~$k\in\Z$, this index is
\[ \ind_p Y=\frac{1}{\alpha}-k.\]

Notice that $k=0$ corresponds to a rotationally symmetric source or
sink of~$Y$, which lifts to an identically looking zero of~$\widetilde{Y}$.
Also, for $\alpha>1$ the term $1-k\alpha$ in the
$\theta$-component of (\ref{eqn:curve}) never equals zero,
no matter what $k\in\Z$, which means that orbifold points
always must be zeros of~$Y$.

\begin{thm}[Poincar\'e--Hopf]
Let $Y$ be a vector field with isolated zeros on a closed, oriented
$2$-dimensional orbifold~$B$.
Then
\[ \sum_{\substack{p\in B\\Y(p)=0}}\ind_pY=\chi_{\mathrm{orb}}(B).\]
\end{thm}

\begin{proof}
The idea is simply to compute the Euler number $e=\chi_{\mathrm{orb}}(B)$
of the Seifert fibration $STB\rightarrow B$ with the help of
an s.o.s.\ $\Sigma^{Y}$ adapted to~$Y$.

Outside small disc neighbourhoods of the zeros of $Y$ we may normalise the
vector field and regard it as a section $\Sigma^Y_0$ of $STB\rightarrow B$
outside this set of discs in~$B$.
This surface $\Sigma^Y_0$ extends to an s.o.s.\ $\Sigma^Y$ of $STB$,
with boundary components certain multiple covers of the
fibres over the zeros $p$ of~$Y$, as in the proof of
Proposition~\ref{prop:seifert-euler}. The multiplicity
of the covering is determined by the number of full turns
the boundary component makes in the $\theta$-direction. Notice that
the orientation of the collection of circles $\pi(\partial\Sigma^Y_0)$
is the opposite of the orientation as boundaries of the removed discs.
Thus, the multiplicity is $-\ind_pY$ at a smooth point and,
by~(\ref{eqn:liftedcurve}), equal to
$-(1-k_i\alpha_i)=-\alpha_i\,\ind_pY$ at an orbifold point
of order~$\alpha_i$, where $k_i\in\Z$ is the integer describing
that particular zero of~$Y$, see also Remark~\ref{rem:intersection}.

With a connection $1$-form $\alpha$ corresponding to regular fibres
having length~$1$ as in the proof of
Proposition~\ref{prop:seifert-euler},
that is, equal to $\rmd\theta/2\pi$ near the fibres over smooth zeros
of~$Y$, and equal to $\rmd\theta/2\pi\alpha_i$ over an orbifold
point of order~$\alpha_i$, we have
\[ \chi_{\mathrm{orb}}(B)=e=-\int_{\Sigma^Y}\rmd\alpha
=-\int_{\partial\Sigma^Y}\alpha
=\sum_{\substack{p\in B\\Y(p)=0}}\ind_pY.\qed\]
\renewcommand{\qed}{}
\end{proof}

\begin{rem}
\label{rem:intersection}
It may be helpful to reformulate the first part of the proof
in terms of meridians and longitudes, similar to the
discussion of the topology of surfaces of section in~\cite{agz19}.

First, consider a zero
of $Y$ at a smooth point $p\in B$. Let $V$ be a tubular
neighbourhood of the fibre $ST_pB$. Let $\mu$ be the meridian on
$\partial V$, and $\lambda$ the longitude determined by
the parallel fibres. We orient $\lambda$ as the fibres, and $\mu$
in such a way that $(\mu,\lambda)$ gives the positive orientation
of $\partial V$. The component of $\partial\Sigma_0^Y$ on
$\partial V$ is $(-1,-\ind_pY)$ in terms of the $(\mu,\lambda)$-basis.
So this component is isotopic to $-\ind_pY$ times the spine of~$V$.
Also, notice that the intersection number of the fibre $(0,1)$
with $(-1,-\ind_pY)$ is~$+1$, which is consistent with $\Sigma_0^Y$ being
a section.

For a zero of $Y$ at a singular point of order~$\alpha_i$,
we take a neighbourhood $V_i$ of $ST_{p_i}B$ with $\mu_i,\lambda_i$
as in Section~\ref{section:seifert}. Now the component of
$\partial\Sigma_0^Y$ on $\partial V_i$ is $(-k_i,-1+k_i\alpha_i)$
by~(\ref{eqn:curve}), which is isotopic to $-1+k_i\alpha_i$ times the spine.
Again, the intersection of the fibre $(-1,\alpha)$, see~(\ref{eqn:gluing}),
with $(-k_i,-1+k_i\alpha_i)$ is~$+1$.
\end{rem}
\section{Transversely holomorphic foliations and the Bott invariant}
\label{section:transhol}
In \cite{gego16} with Jes\'us Gonzalo we proved a generalised
Gau{\ss}--Bonnet theorem
for transversely holomorphic $1$-dimensional foliations on $3$-manifolds.
In certain situations, which I am going to describe now, this can
be interpreted as a statement about $\volX$ for a vector field $X$ whose
flow defines such a foliation.

The following definition is from~\cite{gego95}.

\begin{defn}
A pair of contact forms $(\omega_1,\omega_2)$ on a closed, oriented
$3$-manifold $M$
is called a \emph{Cartan structure} if
\[ \begin{array}{ccccc}
\omega_1\wedge\rmd\omega_1 & = & \omega_2\wedge\rmd\omega_2 & \neq & 0\\
\omega_1\wedge\rmd\omega_2 & = & \omega_2\wedge\rmd\omega_1 & =    & 0.
\end{array} \]
\end{defn}

Such structures exist in abundance, see~\cite[Theorem~1.2]{gego95}.
They are special cases of what we christened taut contact circles
in that paper: any linear combination $\lambda_1\omega_1+\lambda_2\omega_2$
with $(\lambda_1,\lambda_2)\in S^1\subset\R^2$ is again a contact form
defining the same volume form.
The defining equations for a Cartan structure can be rephrased as
saying that there is a uniquely defined nowhere vanishing
$1$-form $\alpha$ such that
\begin{eqnarray*}
\rmd\omega_1 & = & \omega_2\wedge\alpha,\\
\rmd\omega_2 & = & \alpha\wedge\omega_1.
\end{eqnarray*}
In terms of the complex-valued $1$-form $\omegac:=\omega_1+\rmi\,\omega_2$,
these equations can be rewritten as
\[ \rmd\omegac=\rmi\,\alpha\wedge\omegac.\]
Observe that
\[ 0\neq\omega_1\wedge\rmd\omega_1=\omega_1\wedge\omega_2\wedge\alpha,\]
so $\alpha$ is nonzero on the common kernel of $\omega_1$
and~$\omega_2$.

\begin{lem}
\label{lem:cartan}
Let $X$ be the vector field defined by $X\in\ker\omega_1\cap\ker\omega_2$
and $\alpha(X)=1$. Then $i_X\rmd\alpha=0$.
Hence, by Proposition~\ref{prop:wadsley}, $X$ is geodesible.
\end{lem}

\begin{proof}
By taking the exterior derivative of the defining equations for $\alpha$
we find
\[ 0=\rmd^2\omega_1=\rmd\omega_2\wedge\alpha-\omega_2\wedge\rmd\alpha=
-\omega_2\wedge\rmd\alpha,\]
and similarly
\[ 0=\rmd\alpha\wedge\omega_1.\]
This implies that $i_X\rmd\alpha$ must be a multiple both of
$\omega_1$ and~$\omega_2$, but these forms are pointwise
linearly independent.
\end{proof}

The $1$-form $\omegac$ is formally integrable in the sense that
$\omegac\wedge\rmd\omegac=0$. In \cite{gego16} it is shown that
this is equivalent to saying that $\omegac$ defines a transverse
holomorphic structure for the $1$-dimensional
foliation defined by the flow of~$X$. 

In general, the formal integrability of a complex-valued $1$-form
$\omegac$ only implies the existence of a (not
uniquely defined) \emph{complex-valued}
$1$-form $\alphac$ such that
\[ \rmd\omega_c=\alphac\wedge\omegac.\]
A Godbillon--Vey type argument shows that
the complex number
\[ \int_M\alphac\wedge\rmd\alphac,\]
called the \emph{Bott invariant},
is an invariant of the transversely holomorphic foliation that does not
depend on the choice of $\omegac$ and $\alphac$.

The generalised Gau{\ss}--Bonnet theorem \cite[Theorem~3.3]{gego16}
says that for transversely holomorphic foliations coming from
a Cartan structure, this Bott invariant depends only
on the $1$-dimensional foliation defined by the common kernel flow,
not on the specific transverse holomorphic structure.
As observed before, for $\omegac$ coming from a Cartan structure we
can take $\alphac=\rmi\alpha$. Thus, the generalised
Gau{\ss}--Bonnet theorem from \cite{gego16} can be rephrased
as follows.

\begin{thm}
If the vector field $X$ derives from a Cartan structure as described, then
$\volX$ equals the negative of the Bott invariant of any transversely
holomorphic structure on the foliation $\langle X\rangle$.\qed
\end{thm}

The paper \cite{gego16} contains examples which show this to
be a nontrivial statement. There are instances of the generalised
Gau{\ss}--Bonnet theorem where the transverse holomorphic structure
is indeed not unique. In \cite{gego16} one can also find
a complete classification of the transversely holomorphic foliations
on~$S^3$, originally due (for all $3$-manifolds)
to Brunella and Ghys, and a computation of their Bott invariant.
\section{Global surfaces of section}
\label{section:sos}
We now want to compute $\volX$ under the assumption that the geodesible
vector field $X$ admits a global surface of section $\Sigma\subset M$.
For simplicity, we assume that $M$ is a closed, oriented manifold
of dimension~$3$, although our considerations
extend in an obvious manner to global hypersurfaces of section
in manifolds of higher odd dimension for an appropriate
definition of that concept.

Given such an s.o.s., we can associate with each point $p\in\Int(\Sigma)$
its \emph{return time} $\tau (p)\in\R^+$, i.e.\ the smallest positive
real number with $\phi_{\tau(p)}(p)\in\Int(\Sigma)$, were $\phi_t$ denotes
the flow of~$X$.

\begin{prop}
Let $\sigma$ be a basic $2$-form on $M$ that represents the Euler
class~$e_X$. Then
\[ \volX=-\int_{\Int(\Sigma)}\tau\sigma,\]
where we interpret $\sigma$ as a $2$-form on the transversal $\Int(\Sigma)$
for the flow of~$X$.
\end{prop}

\begin{proof}
Let $\alpha$ be a characteristic $1$-form of~$X$.
By Proposition~\ref{prop:pairing} we have
\[ \volX=\int_M\alpha\wedge\rmd\alpha=-\int_M\alpha\wedge\sigma.\]
To compute the integral on the right, we consider the
injective immersion
\[ \begin{array}{rccccc}
\Phi\co & [0,1) & \times & \Int(\Sigma) & \longrightarrow & M\\
        & (t       & ,      & p)           & \longmapsto     &
  \phi_{t\tau(p)}(p).
\end{array} \]
Since $T\Phi(\partial_t)$ is a multiple of~$X$, and $\sigma$ a basic
differential form, we can compute
\begin{eqnarray*}
\int_M\alpha\wedge\sigma
 & = & \int_{M\setminus\partial\Sigma}\alpha\wedge\sigma\\
 & = & \int_{[0,1)\times\Int(\Sigma)}\Phi^*(\alpha\wedge\sigma)\\
 & = & \int_{\Int(\Sigma)}\biggl(\int_0^1\bigl(\Phi^*\alpha\bigr)_{(t,p)}
       (\partial_t)\,\rmd t\biggr)\, \sigma\\
 & = & \int_{\Int(\Sigma)}\tau\sigma.
\end{eqnarray*}
In the last line we used that
\[ \bigl(\Phi^*\alpha\bigr)_{(t,p)}(\partial_t)=
\alpha_{\Phi(t,p)}\bigl(T\Phi(\partial_t)\bigr)=
\alpha_{\Phi(t,p)}\bigl(\tau(p)X\bigr)=\tau(p).\]
Hence $\volX=\displaystyle{-\int_{\Int(\Sigma)}\tau\sigma}$, as claimed.
\end{proof}

\begin{ex}
\label{ex:D2}
On $D^2$ with polar coordinates $(r,\varphi)$ we write
$\lambda=r^2\,\rmd\varphi/2$ for the primitive $1$-form
of the standard area
form $\omega=\rmd\lambda=r\,\rmd r\wedge\rmd\varphi$.
On $\R/\Z\times D^2$ we consider the $1$-form
\[ \alpha=H\,\rmd\theta+\lambda,\]
where $H$ is a smooth function of~$r^2$.
In the sequel it will always be understood that
$H$ or its derivative $H'$ is evaluated at~$r^2$. Then
\[ \rmd\alpha=2rH'\,\rmd r\wedge\rmd\theta+\omega\]
and
\[ \alpha\wedge\rmd\alpha=(H-r^2H')\,\rmd\theta\wedge\omega.\]
We assume that $H-r^2H'>0$; then $\alpha$ is a contact form.
As discussed in~\cite{agz18a}, this $1$-form descends to
a contact form (still denoted~$\alpha$)
on~$S^3$, obtained from $S^1\times D^2$ by
collapsing the circle action on the boundary $S^1\times\partial D^2$
generated by
\[ \partial_{\theta}-2H(1)\partial_{\varphi}
\in\ker\alpha|_{T(S^1\times\partial D^2)}.\]
The Reeb vector field of $\alpha$ (on $S^1\times D^2$) is
\begin{equation}
\label{eqn:reebcut}
X=\frac{\partial_{\theta}-2H'\partial_{\varphi}}{H-r^2H'}.
\end{equation}
Thus,
\[ \volX=\int_{S^3}\alpha\wedge\rmd\alpha=\int_{S^1\times D^2}\alpha\wedge
\rmd\alpha=\int_{D^2}(H-r^2H')\omega.\]
On the other hand, the disc $\{0\}\times D^2$ descends to an s.o.s.\
for the Reeb flow on~$S^3$, and by (\ref{eqn:reebcut})
the return time is
\[ \tau=H-r^2H'.\]
So we see that $\volX$ can likewise be computed as
\[ \volX=\int_{D^2}\tau\sigma\]
with $\sigma=\rmd\alpha|_{TD^2}$ or any other $2$-form that differs
from $\rmd\alpha$ by the differential of a basic $1$-form
for $X$ on~$S^3$ (not just on $S^1\times D^2$).
\end{ex}

\begin{rem}
For an expression of $\volX$ in the preceding example
in terms of the Calabi invariant of the return map on the
s.o.s.\ see \cite{abhs18}.
\end{rem}
\section{Contact forms with the same Reeb vector field}
\label{section:reeb}
In this section we present examples of nondiffeomorphic contact forms
with the same  Reeb vector field.

\begin{thm}
\label{thm:reeb}
In any odd dimension $\geq 9$ there is a closed manifold admitting
a countably infinite family of contact forms that are pairwise
nondiffeomorphic but share the same Reeb vector field.
\end{thm}

\begin{proof}
We construct these manifolds as Boothby--Wang bundles
\cite{bowa58}, \cite[Section~7.2]{geig08} over integral
symplectic manifolds. Starting point for our construction
are examples of symplectic manifolds, in any even dimension $\geq 8$,
with cohomologous but nondiffeomorphic symplectic forms, devised
by McDuff~\cite{mcdu87}. In dimension eight, one begins with
the manifold $S^2\times T^2\times S^2\times S^2$ with the standard split
symplectic form. We think of $T^2$ as $(\R/\Z)^2$.
One then twists this symplectic form by a diffeomorphism
\[ (p_1;s_2,t_2;p_3;p_4)\longmapsto
\bigl(p_1;s_2,t_2,\psi_k(p_1,t_2)(p_3);p_4\bigr),\]
where $\psi_k(p_1,t_2)\co S^2\rightarrow S^2$ is the rotation
of $S^2$ about the axis determined by $\pm p_1$ through an angle
$2\pi kt_2$.
Finally, one takes the symplectic blow-up of these forms along
$S^2\times T^2\times\{(p_3,p_4)\}$ with the same blow-up parameter
(giving the `size' of the blow-up) for all~$k\in\N_0$.

The resulting symplectic forms $\omega_k$ on the blown-up
manifold $W$ are cohomologous and homotopic through
(noncohomologous) symplectic forms, but they are pairwise
nondiffeomorphic.  By taking products with copies of $S^2$, one obtains
similar examples in higher dimensions.

The cohomology class of the symplectic form on a manifold
obtained as a blow-up has been computed in~\cite{mcdu84}, and
from there one sees that
the blow-up can be chosen in such a way that this cohomology class
is rational. Hence, after a constant rescaling we may assume the symplectic
forms $\omega_k$ to be integral, i.e.\ their de Rham cohomology class
$[\omega_k]$ lies in the image of the inclusion $H^2(W;\Z)\subset H^2(W;\R)=
H^2_{\mathrm{dR}}(W)$.

Now choose a class $e\in H^2(W;\Z)$ with $e\otimes\R=-[\omega_k]$,
and let $\pi\co M\rightarrow W$ be the $S^1$-bundle over $W$
of Euler class~$e$. One then finds,
for each $k\in\N_0$, a connection $1$-form $\alpha_k$ on $M$
with curvature form $\omega_k$, that is,
$\rmd\alpha_k=\pi^*\omega_k$, see~\cite[Section~7.2]{geig08}.
(This normalisation corresponds to thinking of $S^1$ as $\R/\Z$.)
Hence, the $\alpha_k$ are contact forms with Reeb vector field given
by the unit tangent vector field along the fibres.

The $\alpha_k$, $k\in\N_0$, are pairwise nondiffeomorphic, because
any diffeomorphism between $\alpha_k$ and $\alpha_{\ell}$ would
preserve the Reeb vector field and hence descend to a diffeomorphism
between $\omega_k$ and~$\omega_{\ell}$.
\end{proof}

\begin{rem}
(1) I do not know whether the contact structures $\ker\alpha_k$
are diffeomorphic. They all have the same underlying almost contact structure.

(2) I hedge my bets concerning dimensions $5$ and~$7$.

(3) Contact forms with all Reeb orbits closed and of the same minimal period
are also called \emph{Zoll contact forms}~\cite{abhs18,abbe19}.
\end{rem}
\section{Orbit equivalence}
\label{section:orbitequ}
A slighty weaker question than the one asked by Viterbo is the following:
are there examples of contact forms with the same Reeb vector field
up to scaling by a function? Or, put differently, one asks
for nondiffeomorphic contact forms whose Reeb flows are smoothly
orbit equivalent. For the more general class of geodesible
vector fields, this problem is best phrased as follows:
on a manifold $M$, is there a geodesible vector field $X$
and a function $f\in C^{\infty}(M,\R^+)$ such that $fX$ is likewise
geodesible? Of course, one should exclude the trivial case of
$f$ being constant, where one simply rescales the metric by
the inverse constant.

This is related, but not equivalent to the question about
nontrivially geodesically equivalent metrics, where two
Riemannian metrics share the same geodesics up to reparametrisation (so the
geodesic flows are orbit equivalent), but one metric is not
a constant multiple of the other.

Matveev~\cite{matv03} has shown that among closed, connected
$3$-manifolds, examples of nontrivially geodesically
equivalent metrics exist only on lens spaces and
Seifert manifolds with Euler number zero. See also~\cite{matv12}
for a discussion of this phenomenon in the context of general relativity.

Our question asks about the nontrivial equivalence of two
foliations by geodesible vector fields. In some sense, this is a weaker
question; on the other hand, a nontrivial equivalence between two
Riemannian metrics may well become trivial when restricted to
any geodesic foliation.

\begin{ex}
On the $2$-torus $T^2=(\R/\Z)^2$ we consider the
standard flat metric $g_1=\rmd x_1^2+\rmd x_2^2$ and a second flat
metric $g_2=\rmd x_1^2+a\,\rmd x_2^2$ with $a\in\R^+\setminus\{1\}$.
Then $g_2$ is not a constant multiple of $g_1$, but the two metrics
are geodesically equivalent: the geodesics in both cases are the
images of straight lines in~$\R^2$ under the projection to~$T^2$.
A geodesic foliation is given by the straight lines of some constant slope,
and along those parallel lines the unit vector fields for the two
metrics differ by a constant.
\end{ex}

Using an idea going back to Beltrami and explained in~\cite{matv03},
we can exhibit a simple example of geodesically equivalent metrics on~$S^3$
that give rise to a geodesible vector field admitting
nontrivial rescalings into likewise geodesible vector fields.
Here, by construction, the vector fields are diffeomorphic.
As I shall explain, rescalings of geodesible vector fields
that define an $S^1$-fibration will always be diffeomorphic.

\begin{ex}
\label{ex:beltrami}
For $a_1,a_2\in\R^+$, consider the linear map
$A=A_{a_1,a_2}\co (z_1,z_2)\mapsto (a_1z_1,a_2z_2)$
on $\C^2=\R^4$. Then define
$\phi=\phi_{a_1,a_2}\co S^3\rightarrow S^3$
by $\phi(p)=A(p)/|A(p)|$. Let $g_{a_1,a_2}=\phi^*g_0$
be the pull-back of the round metric $g_0$ on~$S^3$. Since
$\phi$ takes great circles to great circles, the metric $\phi^*g_0$
is geodesically equivalent to~$g_0$, nontrivially so unless $a_1=a_2$.

A straightforward computation yields the following
expression for $g_{a_1,a_2}$:
\begin{eqnarray*}
g_{a_1,a_2}
& = &  \frac{a_1^2}{\Delta}\,\bigl(\rmd x_1^2+\rmd y_1^2\bigr)
      +\frac{a_2^2}{\Delta}\,\bigl(\rmd x_2^2+\rmd y_2^2\bigr)\\
&   & -\frac{a_1^4}{\Delta^2}\,\bigl(x_1\,\rmd x_1+y_1\,\rmd y_1\bigr)^2
      -\frac{a_2^4}{\Delta^2}\,\bigl(x_2\,\rmd x_2+y_2\,\rmd y_2\bigr)^2\\
&   & -\frac{2a_1^2a_2^2}{\Delta^2}\,\bigl(x_1\,\rmd x_1+y_1\,\rmd y_1\bigr)
       \,\bigl(x_2\,\rmd x_2+y_2\,\rmd y_2\bigr),
\end{eqnarray*}
where we write
\[ \Delta=\Delta_{a_1,a_2}(r_1,r_2)=a_1^2r_1^2+a_2^2r_2^2.\]
Recall that in terms of polar coordinates
we have $\rmd x_i^2+\rmd y_i^2=\rmd r_i^2+r_i^2\,\rmd\varphi_i^2$
and $x_i\,\rmd x_i+y_i\,\rmd y_i=r_i\,\rmd r_i$.

The positive Hopf fibration is generated by $X_0=\partial_{\varphi_1}+
\partial_{\varphi_2}$. This vector field has constant length~$1$
with respect to all the metrics~$g_{a_1,a_2}$, so from the viewpoint
of geodesic foliations this yields nothing new.
Also, the corresponding contact forms
\[ \alpha_{a_1,a_2}=g_{a_1,a_2}(X_0,\,.\,)
=\frac{a_1^2r_1^2\,\rmd\varphi_1+a_2^2r_2^2\,\rmd\varphi_2}{\Delta}\]
all have $X_0$ as Reeb vector field, and so they are just diffeomorphic
deformations of the standard contact form $\alpha_{1,1}$
by Proposition~\ref{prop:three}.

A more interesting choice is to take the great circle foliation generated by
\[ X_1=x_1\partial_{x_2}-x_2\partial_{x_1}+
y_1\partial_{y_2}-y_2\partial_{y_1}.\]
We write $L=L_{a_1,a_2}=\bigl(g_{a_1,a_2}(X_1,X_1)\bigr)^{1/2}$
for the length of $X_1$ with respect to $g_{a_1,a_2}$.
One computes
\[ L^2=\frac{a_1^2r_2^2+a_2^2r_1^2}{\Delta}
-\frac{(a_1^2-a_2^2)^2}{\Delta^2}\,(x_1x_2+y_1y_2)^2.\]

Thus, we have found the nontrivial family of geodesible vector fields
$X_1/L_{a_1,a_2}$, with corresponding metric $g_{a_1,a_2}$, all generating
the same foliation of $S^3$ by great circles.
The corresponding $1$-form
\[ \alpha=\alpha_{a_1,a_2}=g_{a_1,a_2}(X_1/L_{a_1,a_2},\,.\,)\]
can be computed explicitly as
\begin{eqnarray*}
L\alpha
& = & -\frac{a_1^2}{\Delta}\,(x_2\,\rmd x_1+y_2\,\rmd y_1)
      +\frac{a_2^2}{\Delta}\,(x_1\,\rmd x_2+y_1\,\rmd y_2)\\
&   & +\frac{a_1^4-a_1^2a_2^2}{\Delta^2}\,
          (x_1x_2+y_1y_2)\, (x_1\,\rmd x_1+y_1\,\rmd y_1)\\
&   & -\frac{a_2^4-a_1^2a_2^2}{\Delta^2}\,
          (x_1x_2+y_1y_2)\, (x_2\,\rmd x_2+y_2\,\rmd y_2).
\end{eqnarray*}

I did not check whether these are contact forms for all $a_1,a_2\in\R^+$,
but by the openness of the contact condition they certainly are
for $a_1,a_2$ close to~$1$. Then $X_1/L_{a_1,a_2}$
will be the Reeb vector field of $\alpha_{a_1,a_2}$.
\end{ex}

The following proposition gives a more systematic statement about
rescalings of geodesible vector fields that define an $S^1$-fibration.
This is essentially due to Wadsley~\cite{wads75} (in greater
generality); for the case at hand it can be
retraced to the work of Boothby and Wang~\cite{bowa58}.

\begin{prop}
\label{prop:wbw}
Let $X$ be a geodesible vector field on a closed manifold~$M$ such that
the flow lines of $X$ are the fibres of a principal $S^1$-bundle
$M\rightarrow M/S^1$. Then, after a constant rescaling of $X$ all orbits
have (minimal) period~$1$, so that the flow of $X$ defines the $S^1$-action.
A rescaling $fX$ of $X$ is likewise geodesible if and only if
all orbits have the same period. When this period is~$1$,
the vector fields $X$ and $fX$ are diffeomorphic by a diffeomorphism
that sends each fibre to itself and is isotopic to the identity
via such diffeomorphisms.
\end{prop}

\begin{proof}
If $X$ is geodesible, we find a $1$-form $\alpha$ with $\alpha(X)=1$ and
$i_X\rmd\alpha=0$ by Proposition~\ref{prop:wadsley}. Then
\cite[Lemmas 7.2.6 and 7.2.7]{geig08}, which fill a gap
in~\cite{bowa58}, show that the orbits of $X$ all have the same period.
Notice that Lemma~7.2.7 in \cite{geig08} is formulated for Reeb vector fields,
but the proof only uses the property $i_X\rmd\alpha=0$, not the
nondegeneracy of $\rmd\alpha|_{\ker\alpha}$. 

If $fX$ is geodesible, the same argument applies. Conversely,
if the orbits of $fX$ all have the same period, then the flow of $X$
defines an $S^1$-bundle structure, and any connection $1$-form
for this bundle is a characteristic $1$-form for $fX$, which makes $fX$
geodesible.

Now suppose the period of $fX$ equals~$1$. Given a local section $U\cong D^2$
of~$X$, the flow of $X$ defines a trivialisation $U\times S^1$
of the bundle, and $f$ gives rise to a family of $1$-periodic
velocity functions
$v_u\co [0,1]\rightarrow\R^+$ with $\int_0^1v_u(t)\,\rmd t=1$
for every $u\in U$. (I refrain from writing $v_u$
as a function on $S^1$, since the time parameter $t$
should not be confused with the fibre parameter defined by the
flow of~$X$.) Let $\psi\co D^2\rightarrow[0,1]$ be a bump function
equal to $1$ on a disc of radius $1/2$, say, and supported in the interior
of~$D^2$. Then
\[ \mu\bigl(\psi(u)v_u+1-\psi (u)\bigr)+1-\mu \]
defines for each $u\in U$ and $\mu\in [0,1]$ a $1$-periodic
velocity function $[0,1]\rightarrow\R^+$
of integral~$1$. This gives rise to an isotopy along fibres whose
time-$1$ map sends $X$ to $fX$ on fibres where $\psi(u)=1$, and 
which is stationary on fibres along which $f=1$. This allows us to patch
together such local isotopies to obtain the desired result.
\end{proof}

The next corollary also applies to Example~\ref{ex:beltrami}.

\begin{cor}
If $R$ and $fR$ are Reeb vector fields on a closed $3$-manifold
with all orbits periodic of the same period~$1$, then any corresponding
contact forms are related via a fibre-preserving isotopy.
\end{cor}

\begin{proof}
This follows immediately by combining the last statement of
Proposition~\ref{prop:wbw} with Proposition~\ref{prop:three}.

Alternatively, one can give a direct proof, using a refinement of
the proof of Proposition~\ref{prop:three}.
Let $\alpha_0$ be a contact form with Reeb vector field~$R$, and
$\alpha_1$ a contact form for~$fR$. Set $\alpha_t:=(1-t)\alpha_0+t\alpha_1$.
We would like to find an isotopy $(\psi_t)_{t\in[0,1]}$ satisfying
(\ref{eqn:isotopy3}) as in the proof of Proposition~\ref{prop:three}.

The $\alpha_t$ are contact forms with Reeb vector field $R_t$
proportional to~$R$. We try to find an isotopy $(\psi_t)$ generated by
a vector field $X_t$ of the form
\[ X_t= h_tR_t+Y_t\]
with $Y_t\in\ker\alpha_t$. Differentiating (\ref{eqn:isotopy3}) we find
\begin{equation}
\label{eqn:Y}
\alpha_1-\alpha_0+\rmd h_t+i_{Y_t}\rmd\alpha_t=0.
\end{equation}
When we plug $R$ into this equation, we find
\begin{equation}
\label{eqn:g}
f^{-1}-1+\rmd h_t(R)=0.
\end{equation}
The condition that the period of $fR$ be $1$ translates into $f^{-1}$
integrating to $1$ along any fibre of the $S^1$-bundle. This allows us
to define a family of functions $h_t$ satisfying (\ref{eqn:g}),
and then there is a unique vector field $Y_t\in\ker\alpha_t$
satisfying~(\ref{eqn:Y}). Both $\alpha_1-\alpha_0+\rmd h_t$ and
$\rmd\alpha_t$ are lifts of differential forms on the quotient
surface $M/S^1$, hence the flow of $Y_t$ preserves fibres.
\end{proof}

\begin{rem}
If $R$ is the Reeb vector field of a contact form~$\alpha$
(on a connected manifold~$M$), then
the rescaled vector field $f^{-1}R$ is never the Reeb vector field of
$f\alpha$, unless the function~$f$ is constant, for the identity
\[ 0=i_R\rmd(f\alpha)=i_R(\rmd f\wedge\alpha)=
\rmd f(R)\alpha-\rmd f\]
implies that $\rmd f$ vanishes on all vectors
tangent to the contact structure $\ker\alpha$,
and by \cite[Theorem~3.3.1]{geig08} any two points in $M$ can be joined
by a curve tangent to the contact structure.
\end{rem}

\begin{ack}
Special thanks go to Claude Viterbo for putting this question to me, and
to Alberto Abbondandolo for showing me identity~(\ref{eqn:alberto}).
I first learned about basic cohomology from a talk by Robert Wolak.
I thank Peter Albers, Jes\'us Gonzalo and Kai Zehmisch
for their collaboration on various projects that have
contributed in many ways to the conception of this note,
and Christian Lange for useful conversations on $S^1$-orbibundles.
I also thank the referee for perceptive and detailed comments on the
original manuscript.
\end{ack}

\end{document}